\documentclass[a4paper, 11pt]{article}
\usepackage[mathscr]{eucal}
\usepackage{amssymb}
\usepackage{latexsym}
\usepackage{amsthm}
\usepackage{amsmath}
\usepackage[dvips]{graphicx}
\usepackage{psfrag}
\usepackage{a4wide}

\newcommand{\x}{\boldsymbol{x}}
\newcommand{\y}{\boldsymbol{y}}

\newtheorem{theorem}{Theorem}[section]

\newtheorem{corollary}[theorem]{Corollary}
\theoremstyle{definition}

\newtheorem{remark}[theorem]{Remark}

\newtheorem{example}[theorem]{Example}

\makeatletter
\@addtoreset{equation}{section}

\makeatother

\title{Bounds for sets with few distances distinct modulo a prime ideal}
\author{
Hiroshi Nozaki\thanks{Department of Mathematics Education, 
	Aichi University of Education, 
	1 Hirosawa, Igaya-cho, 
	Kariya, Aichi 448-8542, 
	Japan. {\tt hnozaki@auecc.aichi-edu.ac.jp}}}

\begin{document}

\maketitle

\renewcommand{\thefootnote}{\fnsymbol{footnote}}
\footnote[0]{2010 Mathematics Subject Classification: 
05D05 (05B30)
}

\begin{abstract}
Let $\mathcal{O}_K$ be the ring of integers of an algebraic number field $K$ embedded into $\mathbb{C}$. 
Let $X$ be a subset of the Euclidean space $\mathbb{R}^d$, and $D(X)$ be the set of the squared distances of two distinct points in $X$. 
In this paper, 
we prove that if $D(X)\subset \mathcal{O}_K$ and there exist $s$ values $a_1,\ldots, a_s \in \mathcal{O}_K$ distinct modulo a prime ideal $\mathfrak{p}$ of $\mathcal{O}_K$ such that each $a_i$ is not zero modulo $\mathfrak{p}$ and each element of $D(X)$ is congruent to some $a_i$, then  $|X| \leq \binom{d+s}{s}+\binom{d+s-1}{s-1}$.   
\end{abstract}

\textbf{Key words}: 
$s$-distance set, algebraic number field

\section{Introduction}
This paper is devoted to giving an upper bound on the cardinalities of certain  finite sets $X$ in a metric space $M$, that have some special properties of the values of distances appearing in $X$. A finite set $X$ in $M$ is called an {\it $s$-distance set} if the number of distances of two distinct points in $X$ is equal to $s$. One of the major problems for $s$-distance sets is to determine the largest possible $s$-distance sets for given $s$, that is motivated from the extremal set theory. For this purpose, we need to give (or improve) upper bounds on the size and construct large sets. 
For small $s$, there are remarkable successful cases that can determine the largest sets, for example,  $2$-distance sets on a Euclidean sphere \cite{GY18}, sets of equiangular lines in Euclidean spaces \cite{JTYZZ21}, and several $s$-distance sets in the real, complex, or quaternionic projective spaces \cite{L92}  or in polynomial association schemes \cite{D73,DGS77,DL98}.  

In the literature of combinatorial geometry, 
upper bounds for $s$-distance sets for $L$-intersecting families have been obtained. 
 An {\it $L$-intersecting family} $\mathfrak{F}$ is a family of subsets of a finite set $F$
that satisfies $|A \cap B| \in L$ for any distinct $A,B \in \mathfrak{F}$ for some $L\subset \{0,1,\ldots, n-1\}$, where $|F|=n$. An $L$-intersecting family $\mathfrak{F}$ is said to be {\it $k$-uniform} if $|A|=k$ for each $A \in \mathfrak{F}$ for some constant $k$. 
For $k$-uniform $L$-intersecting families $\mathfrak{F}$, 
Ray-Chaudhuri and Wilson \cite{R-CW75} proved an upper bound $|\mathfrak{F}| \leq \binom{n}{s}$, where $|L|=s$. This case corresponds to $s$-distance sets in Johnson association schemes. After this work, Frankl and Wilson \cite{FW81} obtained $|\mathfrak{F}|\leq \sum_{i=0}^s\binom{n}{i}$ without the assumption of $k$-uniform. 

Frankl and Wilson \cite{FW81} also proved a {\it modular version} of the upper bound for $k$-uniform $L$-intersecting families. Namely they interpreted the sizes of the intersections as elements of $\mathbb{Z}/p \mathbb{Z}$ for some prime number $p$. Suppose the set $L$ has only $s$ elements distinct modulo $p$, and note that $|L|$ may be greater than $s$.  
For $k$-uniform $L$-intersecting families $\mathfrak{F}$, if  
$k \not \equiv a \pmod{p}$ for each $a \in L$, then 
Frankl--Wilson  \cite{FW81} showed that $|\mathfrak{F}| \leq \binom{n}{s}$. Note that this is the same upper bound as that obtained under the assumption $|L|=s$.  
For $L$-intersecting families $\mathfrak{F}$ with $r$ different sizes of elements of $\mathfrak{F}$ modulo $p$, 
Alon, Babai, and Suzuki \cite{ABS91} proved that $|\mathfrak{F}|\leq  \sum_{i=0}^{r-1}\binom{n}{s-i}$ under a certain weak assumption which is simplified by \cite{HK15}. 
The upper bound in \cite{ABS91} is proved by Koornwinder's method \cite{K77}, which gives upper bounds on the size $|X|$ by proving the linear independence of some polynomial functions that have a bijective correspondence to $X$. 

We have upper bounds for Euclidean $s$-distance sets with several conditions,  which are counterparts of that of $L$-intersecting families. Let $X$ be an $s$-distance set in the Euclidean space $\mathbb{R}^d$.  
For $X$ in the unit sphere $S^{d-1}$, which corresponds to the condition of $k$-uniform, Delsarte, Goethals, and Seidel \cite{DGS77} proved that $|X| \leq \binom{d+s-1}{s}+\binom{d+s-2}{s-1}$. 
With no assumption, Bannai, Bannai, and Stanton \cite{BBS83} proved that $|X| \leq \binom{d+s}{s}$. 
For $X$ in $r$ concentric spheres, which corresponds to the condition of $r$ different sizes, Bannai, Kawasaki, Nitamizu, and Sato \cite{BKNS03}
obtained that $|X| \leq \sum_{i=0}^{2r-1}\binom{d+s-1-i}{s-i}$. 
Recently simple alternative proofs of these upper bounds are given in \cite{HR21,PP21}. 

Blokhuis \cite{BT} gave a modular version of upper bounds for Euclidean sets, assuming the squared distances are rational integers. Let $D(X)$ be the set of the squared Euclidean distances between two distinct points of $X$. 
\begin{theorem}[{mod-$p$ bound \cite{BT}}]  \label{thm:modp}
Let $X$ be a subset of $\mathbb{R}^d$, and $p$ a prime number. 
Suppose $D(X)$ is a subset of rational integers $\mathbb{Z}$. 
If there exist $a_1,\ldots, a_s \in \mathbb{Z}$ distinct modulo $p$ such that 
\begin{enumerate}
    \item for each $i \in \{1,\ldots,s\}$, $a_i \not \equiv 0 \pmod{p}$ and 
    \item for each $\alpha \in D(X)$, there exists $i \in \{1,\ldots,s\}$ such that $\alpha \equiv a_i \pmod{p}$,
\end{enumerate}
then 
\[
|X| \leq \binom{d+s}{s}+\binom{d+s-1}{s-1}. 
\]
\end{theorem}
For sets in a sphere, several projective spaces, or $Q$-polynomial association schemes,  Theorem~\ref{thm:modp} can be analogously obtained. However, the $r$-concentric spherical version is still open. 
The assumption $D(X) \subset \mathbb{Z}$ in Theorem~\ref{thm:modp} is surely strong, and the sets to 
which the theorem can be applied are restricted. In this paper, we extend Theorem~\ref{thm:modp} to  
the ring of integers $\mathcal{O}_K$ of an algebraic number field $K$, and any prime ideal $\mathfrak{p}$ of it. Note that throughout this paper, we fix an embedding of $K$ into $\mathbb{C}$, and $K$ is interpreted as a subfield of $\mathbb{C}$. 
We use Koornwinder's method to prove this mod-$\mathfrak{p}$ upper bound. However the method  
 to prove the linear independence of polynomial functions is  new. 
In the proof, the localization $A_{\mathfrak{p}}$ of $A=\mathcal{O}_K$ by a prime ideal $\mathfrak{p}$ plays a key role and Nakayama's lemma is applied for a certain finitely generated $A_{\mathfrak{p}}$-module.  
This method is purely algebraic and uniformly applicable to 
the polynomial spaces \cite{G89}, \cite[Sections 14--16]{Gb}, which includes the Euclidean sphere, the real, complex or quaternionic projective spaces (see \cite{DL98,L92} for the theory of $s$-distance set in these projective spaces), or $Q$-polynomial association schemes (which include the theory of $L$-intersecting family as codes of Johnson or Hamming schemes) \cite{D73}. 

The paper is organized as follows. 
In Section \ref{sec:2}, 
we introduce basic terminology and results about algebraic number fields. 
In Section \ref{sec:3},
we prove a generalization of Theorem~\ref{thm:modp} (mod-$\mathfrak{p}$ bound) for 
the ring of integers $\mathcal{O}_K$ of an algebraic number field $K$ and a prime ideal $\mathfrak{p}\subset \mathcal{O}_K$. 
We also comment on the version of the theorem for an ideal that may not be prime.  
In Section~\ref{sec:4}, 
we extend the LRS type theorem proved in \cite{LRS77,N11}. 
Namely, if the cardinality of an $s$-distance set is relatively large, 
then a certain ratio of squared distances must be an algebraic integer (see Theorem~\ref{thm:LRS}). 
We explain the relationship between the LRS type theorem and mod-$\mathfrak{p}$ bound, which can refine an upper bound on the size of an $s$-distance set with given distances.

\section{Preliminaries} \label{sec:2}
An extension field $K$ of rationals $\mathbb{Q}$ is an {\it algebraic number field} if the degree $[K:\mathbb{Q}]$ is finite. 
An algebraic number field $K$ can be embedded into $\mathbb{C}$, and $K$ is always identified with a fixed specific subfield of $\mathbb{C}$. 
The {\it ring of integers $\mathcal{O}_K$} is the ring consisting of all algebraic integers in $K$, where an {\it algebraic integer} is a complex number which is a root of a monic polynomial with integer coefficients.   
It is well known that $K$ is the quotient field of $\mathcal{O}_K$, a prime ideal of $\mathcal{O}_K$ is maximal, $\mathcal{O}_K$ is a finitely generated free $\mathbb{Z}$-module,  and $\mathcal{O}_K$ may not be a principal ideal domain. 
For easy examples, if $K=\mathbb{Q}$, then $\mathcal{O}_K=\mathbb{Z}$. 
If $K=\mathbb{Q}(\sqrt{d})$ for a square-free integer $d$, then 
\[
\mathcal{O}_K=\begin{cases}
\mathbb{Z}+ \frac{1+\sqrt{d}}{2} \mathbb{Z} \text{\quad  if $ d \equiv 1 \pmod{4}$},\\ 
\mathbb{Z}+ \sqrt{d} \mathbb{Z} \text{\quad if $ d \equiv 2,3 \pmod{4}$}.
\end{cases}
\]

Suppose a ring is commutative and contains the identity. 
For a ring $A$, $(A,\mathfrak{m})$ is a {\it local ring} if $A$ has a unique maximal ideal $\mathfrak{m}$. 
It is well known that for a ring $A$ and its maximal ideal $\mathfrak{p}\subset A$, we can construct a local ring $(A_{\mathfrak{p}}, \mathfrak{p} A_{\mathfrak{p}})$.  
For $A=\mathcal{O}_K$, the local ring is 
\[
A_\mathfrak{p}=S^{-1} A=\{a/s \in K \mid a \in A, s \in S\},
\]
where $S=A \setminus \mathfrak{p}$. 
Its unique maximal ideal is $\mathfrak{p} A_{\mathfrak{p}}$, which is the ideal of 
$A_{\mathfrak{p}}$ generated by the elements of $\mathfrak{p}$. Note that $A_\mathfrak{p}$ is a principal ideal domain. 
The natural map $f: A/\mathfrak{p} \rightarrow A_\mathfrak{p}/ \mathfrak{p} A_\mathfrak{p}$ is a field isomorphism. 

The following theorem is called Nakayama's lemma, which plays a key role in a proof of the main theorem. For a local ring $(A,\mathfrak{m})$, the ideal $I$ in Theorem~\ref{thm:nakayama} is $\mathfrak{m}$. 
\begin{theorem}\label{thm:nakayama}
Let $A$ be a ring. Let $I$ be an ideal that is contained in all maximal ideals of $A$. Let $M$ be a finitely generated $A$-module. 
If $IM=M$, then $M=\{0\}$. 
\end{theorem}

In order to prove the main theorem, we use the polynomial 
\[
        f_{\x}(\boldsymbol{\xi})=\prod_{i=1}^s (||\x-\boldsymbol{\xi}||^2-a_i),
\]
for $\x \in \mathbb{R}^d$, $a_i \in \mathbb{R}$, and variables $\boldsymbol{\xi}=(\xi_1,\ldots, \xi_d)$, where $||\x||$ is the Euclidean norm of $\x$. By proving the linear independence of $\{f_{\x}\}_{\x \in X}$ as polynomial functions, the cardinality $|X|$ can be bounded above by the dimension of a certain linear space that contains $f_{\x}$. We use the same polynomial space  used in Bannai--Bannai--Stanton \cite{BBS83}. 
For $\xi_0=\xi_1^2+\cdots + \xi_d^2$, we define the polynomial space $P_s(\mathbb{R}^d)$ that consists of all real polynomial functions on $\mathbb{R}^d$ which are spanned by $\xi_0^{\lambda_0}\xi_1^{\lambda_1}\cdots \xi_d^{\lambda_d}$ with $\sum_{i=0}^d\lambda_i\leq s$.  
The dimension of $P_s(\mathbb{R}^d)$ is equal to $\binom{d+s}{s}+\binom{d+s-1}{s-1}$. 

\section{Bounds for $s$-distance sets modulo $\mathfrak{p}$} \label{sec:3}
The following is the main theorem in this paper. 
\begin{theorem}[mod-$\mathfrak{p}$ bound] \label{thm:main}
Let $X$ be a subset of $\mathbb{R}^d$, and $A=\mathcal{O}_K$ the ring of integers of an algebraic number field $K$. 
Let $\mathfrak{p}$ be a prime ideal of $A$. 
Suppose $D(X) \subset A_{\mathfrak{p}}$. 
If there exist $a_1,\ldots, a_s \in A_\mathfrak{p}$ distinct modulo $\mathfrak{p}A_{\mathfrak{p}}$ such that 
\begin{enumerate}
    \item for each $i \in \{1,\ldots,s\}$, $a_i \not \equiv 0 \pmod{\mathfrak{p}A_\mathfrak{p}}$ and 
    \item for each $\alpha \in D(X)$, there exists $i \in \{1,\ldots,s\}$ such that $\alpha \equiv a_i \pmod{\mathfrak{p}A_{\mathfrak{p}}}$,
\end{enumerate}
then 
\[
|X| \leq \binom{d+s}{s}+\binom{d+s-1}{s-1}. 
\]
\end{theorem}
\begin{proof}
        For each $\x \in X$, we define the polynomial  $f_{\x}(\boldsymbol{\xi})\in P_s(\mathbb{R}^d)$   as 
        \[
        f_{\x}(\boldsymbol{\xi})=\prod_{i=1}^s (||\x-\boldsymbol{\xi}||^2-a_i),
        \]
        where if needed, we replace $a_i$ with a real value equivalent to $a_i$ modulo $\mathfrak{p}A_{\mathfrak{p}}$. 
        These polynomials satisfy 
        \begin{equation} \label{eq:f_x(x)}
        f_{\x}(\boldsymbol{x})=(-1)^s \prod_{i=1}^s a_i \not\equiv 0 \pmod{\mathfrak{p}A_{\mathfrak{p}}}, 
        \end{equation}
        and
        \begin{equation} \label{eq:f_x(y)}
        f_{\x}(\boldsymbol{y}) \equiv 0 \pmod{\mathfrak{p}A_{\mathfrak{p}}}
        \end{equation}
        for $\x \ne \y \in X$. 
        
        We prove $\{f_{\x}\}_{\x \in X}$ is linearly independent as  polynomial functions on  $\mathbb{R}^d$.
        Assume there exist $m_{\x} \in \mathbb{R}$ such that 
        \begin{equation} \label{eq:=0}
            \sum_{\x \in X} m_{\x} f_{\x}(\boldsymbol{\xi})=0.
        \end{equation} 
        Let $M$ be an $A_\mathfrak{p}$-module
        generated by a finite set $\{m_{\x} \}_{\x \in X}$, namely
        \[
        M=\sum_{\x \in X} m_{\x} A_\mathfrak{p}.       \]
        From equalities \eqref{eq:=0} and \eqref{eq:f_x(y)}, for each $\y \in X$, 
        \[
        m_{\y} f_{\y}(\y)= - \sum_{\y \ne \x \in X} m_{\x} f_{\x}(\y) \in \mathfrak{p}A_{\mathfrak{p}}M.
        \]
        Since $f_{\y}(\y) \in A_{\mathfrak{p}} \setminus \mathfrak{p}A_{\mathfrak{p}}$ from equality \eqref{eq:f_x(x)}, it follows that $f_{\y}(\y) \in A_{\mathfrak{p}}^\times$ and  
        \[
         m_{\y} = - \sum_{\y \ne \x \in X} m_{\x} f_{\x}(\y)(f_{\y}(\y))^{-1} \in \mathfrak{p}A_{\mathfrak{p}}M. 
        \]
        This implies that $M \subset \mathfrak{p}A_{\mathfrak{p}}M$, and hence 
        $M = \mathfrak{p}A_{\mathfrak{p}}M$. 
        By Nakayama's lemma, $M=\{0\}$ and $m_{\x}=0$ for each $ \x \in X$. 
        Therefore $\{f_{\x}\}_{\x \in X}$ is linearly independent, and 
        \[
        |X| =|\{f_{\x}\}_{\x \in X}| \leq \dim P_s(\mathbb{R}^d)=\binom{d+s}{s} +\binom{d+s-1}{s-1}
        \]
        as desired. 
\end{proof}
\begin{corollary} \label{coro:modp}
Let $X$ be a subset of $\mathbb{R}^d$, and $\mathcal{O}_K$ the ring of integers of an algebraic number field $K$. 
Let $\mathfrak{p}$ be a prime ideal of $\mathcal{O}_K$. 
Suppose $D(X) \subset\mathcal{O}_K$. 
If there exist $a_1,\ldots, a_s \in \mathcal{O}_K$ distinct modulo $\mathfrak{p}$ such that 
\begin{enumerate}
    \item for each $i \in \{1,\ldots,s\}$, $a_i \not \equiv 0 \pmod{\mathfrak{p}}$ and 
    \item for each $\alpha \in D(X)$, there exists $i \in \{1,\ldots,s\}$ such that $\alpha \equiv a_i \pmod{\mathfrak{p}}$,
\end{enumerate}
then 
\[
|X| \leq \binom{d+s}{s}+\binom{d+s-1}{s-1}. 
\]
\end{corollary}
\begin{proof}
Since $A=\mathcal{O}_K \subset A_{\mathfrak{p}}$ and $A/\mathfrak{p} \cong A_\mathfrak{p}/\mathfrak{p} A_{\mathfrak{p}}$, 
this corollary is immediate from Theorem~\ref{thm:main}. 
\end{proof}
\begin{example}
For $X=\{(0,0),(1,0), (-\sqrt{3}/2,1/2),(-\sqrt{3}/2,-1/2)\} \subset \mathbb{R}^2$, the squared distances are $D(X)=\{1,2+\sqrt{3}\}$. 
We take the algebraic number field $K=\mathbb{Q}(\sqrt{3})$. 
Then the ring of integers is $\mathcal{O}_K=\mathbb{Z}+\sqrt{3}\mathbb{Z}$, and $\mathfrak{p}=(1+\sqrt{3})$ is a prime ideal of $\mathcal{O}_K$. 
Since $1\equiv 2+\sqrt{3} \pmod{\mathfrak{p}}$ holds, 
we have $|X| \leq \binom{d+1}{1}+\binom{d}{0}=4$. The set $X$ is an example attaining the upper bound in Corollary~\ref{coro:modp}.
\end{example}

We can prove a similar theorem to Theorem~\ref{thm:main} for an ideal $I \subset \mathcal{O}_K$ which may not be prime as follows. 
\begin{theorem}\label{thm:any_ideal}
Let $X$ be a subset of $\mathbb{R}^d$, and $A=\mathcal{O}_K$ the ring of integers of an algebraic number field $K$. 
Let $I$ be an ideal of $A$, and $I=\mathfrak{p}_1^{\lambda_1} \cdots \mathfrak{p}_r^{\lambda_r}$ the prime decomposition of $I$. 
Let $A_{I}=S^{-1}A=\{a/s \mid a \in A, s \in S\}$, where $S=\bigcup_{R \in (A/I)^\times} R$.  
Suppose $D(X) \subset A_{I}$. 
If there exist $a_1,\ldots, a_s \in A_I$ distinct modulo $I A_{I}$ such that 
\begin{enumerate}
    \item for each $i \in \{1,\ldots,s\}$, $a_i \in A_I^\times$ and 
    \item for each $\alpha \in D(X)$, there exists $i \in \{1,\ldots,s\}$ such that $\alpha \equiv a_i \pmod{I A_{I}}$,
\end{enumerate}
then 
\[
|X| \leq \binom{d+s}{s}+\binom{d+s-1}{s-1}. 
\]
\end{theorem}
\begin{proof}
The proof is similar to that of Theorem~\ref{thm:main}, but we use 
$\mathfrak{p}_1\cdots \mathfrak{p}_r A_I$ instead of $\mathfrak{p} A_\mathfrak{p}$ as the ideal that is contained in  all maximal ideals in Nakayama's lemma.  
\end{proof}
For Theorem \ref{thm:any_ideal}, we must choose squared distances $a_i$ from $A_I^\times$. Such distances $a_i$ can be expressed by $a_i=s_1/s_2$ for some $s_1,s_2 \in S=\bigcup_{R \in (A/I)^\times} R$. 
Since $S \subset \bigcup_{R \in (A/\mathfrak{p}_j)^\times} R$ for any $j$, the squared distances $a_i$ are also elements of $A_{\mathfrak{p}_j}^\times=A\setminus \mathfrak{p}_j$. 
The natural homomorphisms 
\begin{align*}
A_I/IA_I &\rightarrow A_I/\mathfrak{p}_1^{\lambda_1} A_I \times 
\cdots \times A_I/ \mathfrak{p}_r^{\lambda_r}A_I\\
&\rightarrow A_I/\mathfrak{p}_1 A_I \times 
\cdots \times A_I/ \mathfrak{p}_r A_I\\
&\rightarrow A_{\mathfrak{p}_1}/\mathfrak{p}_1 A_{\mathfrak{p}_1} \times 
\cdots \times A_{\mathfrak{p}_r}/ \mathfrak{p}_r A_{\mathfrak{p}_r}
\end{align*}
imply that the number of squared distances distinct modulo $IA_I$ is greater than or equal to that modulo $\mathfrak{p}_i A_{\mathfrak{p}_i}$ for any $i\in \{1,\ldots, r\}$.  
Therefore, Theorem \ref{thm:main} corresponding to the prime-ideal version gives the strongest upper bound for any ideal under our condition.   
\section{LRS type theorem}  \label{sec:4}
We now generalize the LRS type theorem proved in \cite{N11} as follows. 
The absolute bound $|X|\leq \binom{d+s}{s}$ is improved by this generalization. 
\begin{theorem} \label{thm:LRS}
Suppose $s\geq 2$. 
Let $X$ be an $s$-distance set in $\mathbb{R}^d$ and $N=\dim P_{s-1}(\mathbb{R}^d)=\binom{d+s-1}{s-1}+\binom{d+s-2}{s-2}$. 
If $|X| \geq N+(N+1)/t$ for some $t \in \mathbb{N}$, then
\begin{equation*} \label{eq:K_j}
K_j=\prod_{i=1, i\ne j}^s \frac{\alpha_i}{\alpha_i-\alpha_j}
\end{equation*}
is an algebraic integer of degree at most $t$ for each $j \in \{1,\ldots,s \}$. 
\end{theorem}
\begin{proof}
Fix $j \in \{1,\ldots,s\}$. Define the polynomial 
\[
f(\x, \boldsymbol{\xi})=\prod_{i=1,i\ne j}^s \frac{\alpha_i-||\x-\boldsymbol{\xi}||^2}{\alpha_i-\alpha_j}
\]
for each $\x \in X$. 
Since $f(\x,\boldsymbol{\xi}) \in P_{s-1}(\mathbb{R}^d)$, 
the rank of the matrix $M=(f(\x,\y))_{\x,\y \in X}$ is at most $N$ \cite{N11}.  
The matrix can be expressed by
\[
M=K_j I + A_j,
\]
where $I$ is the identity matrix and $A_j$ is a $(0,1)$-matrix with off diagonals. 
Since the size of $M$ is at least $N+(N+1)/t>N$, the matrix has $0$ eigenvalue whose multiplicity is at least $(N+1)/t$. This implies $-K_j$ is the eigenvalue of $A_j$, and hence $K_j$ is an algebraic integer. 

Assume $K_j$ is an algebraic integer of degree larger than $t$. 
Then the number of the conjugates of $-K_j$ is at least $t$, and the conjugates are also eigenvalues of $A_j$. 
Since $A_j$ has the eigenvalue $-K_j$ with multiplicity at least $(N+1)/t$, the size of $A_j$ is at least $(t+1)(N+1)/t=N+1+(N+1)/t$, which contradicts our assumption. 
Therefore $K_j$ is an algebraic integer of degree at most $t$. 
\end{proof}

For $t=1$, the values $K_j$ are  integers under the condition in Theorem~\ref{thm:LRS}, which is the previous result proved in \cite{N11}. 
The following corollaries are immediate from Theorem~\ref{thm:LRS}.
\begin{corollary}
If $K_j$  is not an algebraic integer for some $j \in \{1,\ldots, s\}$, then $|X| \leq N$.  
\end{corollary}
\begin{corollary}\label{coro:imp}
Suppose $K_j$ is an algebraic integer for each $j \in \{1,\ldots,s\}$. 
Let $t$ be the maximum value of the degrees of $K_j$. 
If $t>1$ holds, then $|X| < N+(N+1)/(t-1)$. 
\end{corollary}
Corollary \ref{coro:imp} is an improvement of the absolute bound for $s$-distance sets with the LRS ratios.

If there exist $\alpha_i,\alpha_j \in D(X)\subset \mathcal{O}_K$ such that $\alpha_i$ is congruent to $\alpha_j$ modulo some prime ideal $\mathfrak{p}$ and $\alpha \not \equiv 0 \pmod{\mathfrak{p}}$ for each $\alpha \in D(X)$, then the LRS ratio $K_j$ is not an algebraic integer. Indeed, if $K_j \in \mathcal{O}_K$, then 
\begin{equation} \label{eq:no}
    0\equiv K_j\prod_{i\ne j}(\alpha_i-\alpha_j) = \prod_{i \ne j} \alpha_i \not \equiv 0 \pmod{\mathfrak{p}},
\end{equation}
which is a contradiction. 
When $K_j$ is not an algebraic integer for some $j$, we obtain the bound $|X|\leq N$ by Theorem~\ref{thm:LRS}, and we may obtain a better bound depending on the number of the elements of $D(X)$ distinct modulo $\mathfrak{p}$. 

The results proved in this paper -- mod-$\mathfrak{p}$ bound and LRS type theorem-- are analogously obtained for the sphere $S^{d-1}$ \cite{DGS77}, several projective spaces \cite{DL98,L92}, or $Q$-polynomial association schemes \cite{D73,DL98}. 
For spherical case, the LRS type theorem with $\mathcal{O}_K=\mathbb{Z}$ is useful to determine largest spherical $s$-distance sets for $s=2,3$. In \cite{GY18,MN11}, 
several largest $s$-distance sets are determined by a computer assistance.
The possibilities of choices of integers $K_i$ are finite, and we can take the finite choices of distances from $K_i$. Reducing the number of the possible distances is 
helpful to cut the computational cost by a computer. 
However, Equation \eqref{eq:no} implies that it is impossible to reduce the choices of distances by our results. 
\begin{remark}
Akihiro Munemasa, one of the editors of the journal, communicated to the author the following idea to prove
Theorem \ref{thm:main} without the use of Nakayama's lemma. 
Let $f_{\x}(\boldsymbol{\xi})$ be the same as in the proof of Theorem~\ref{thm:main}. 
We consider the matrix $M=(f_{\x} (\y))_{\x,\y \in X}$, where $X$ satisfies the condition of the theorem. 
In order to prove the linear independence of $\{f_{\x} \}_{\x \in X}$, it suffices to show that the determinant of $M$ is non-zero. The entries of $M$ are elements of $A_\mathfrak{p}$, 
and $M$ is congruent to some diagonal matrix modulo $\mathfrak{p} A_{\mathfrak{p}}$ whose diagonal entries are units in $A_\mathfrak{p}$. The determinant $M$ is not congruent to 0 modulo $\mathfrak{p} A_{\mathfrak{p}}$, in particular, it is non-zero.   
\end{remark}

\bigskip

\noindent
\textbf{Acknowledgments.} 
The author thanks Akihiro Munemasa for providing the idea of an alternative proof of Theorem \ref{thm:main} as editor's comments. 
The author is supported by JSPS KAKENHI Grant Numbers 18K03396, 19K03445,  20K03527, and 22K03402.

\end{document}